\documentclass[12pt,twoside,letter]{amsart}
\usepackage{mathrsfs}

\usepackage{txfonts}
\usepackage{bbm}

\usepackage{amsfonts}
\usepackage{amsmath}
\usepackage{amssymb}
\usepackage{wasysym}
\usepackage{psfrag}
\usepackage{graphics,epsfig,amsmath}  
\usepackage{comment}
\usepackage{xcolor}
\usepackage{enumerate}

\newtheorem{theorem}{Theorem}[section]
\newtheorem{definition}[theorem]{Definition}
\newtheorem{remark}[theorem]{Remark}
\newtheorem{lemma}[theorem]{Lemma}
\newtheorem{prop}[theorem]{Proposition}
\newtheorem{corollary}[theorem]{Corollary}

\newtheorem{example}[theorem]{Example}

\newcommand{\al}{\alpha}

\newcommand{\lb}{\lambda}

\newcommand{\N}{\mathbb{N}}

\newcommand{\R}{\mathbb{R}}

\newcommand{\om}{\omega}
\newcommand{\Om}{\Omega}

\newcommand{\cO}{\mathcal{O}}

\newcommand{\cA}{\mathcal{A}}
\newcommand{\cF}{\mathcal{F}}
\newcommand{\cL}{\mathcal{L}}

\newcommand{\cK}{\mathcal{K}}
\newcommand{\cU}{\mathcal{U}}

\newcommand{\cY}{\mathcal{Y}}

\newtheorem*{thmA}{Theorem A}
\newtheorem*{thmB}{Theorem B}

\title[Dynamics of dissipative contact Hamiltonian systems]{Global dynamics of contact Hamiltonian systems (I): monotone systems}

\author[L. Jin]{Liang Jin}
\address{Department of Mathematics, Nanjing University of Science and Technology,
Nanjing 210094, China}
\email{jl@njust.edu.cn}

\author[J. Yan] {Jun Yan}
\address{School of Mathematical Sciences, Fudan University,
Shanghai 200433, China}
\email{yanjun@fudan.edu.cn}

\thanks{L. Jin is supported in part by the National Natural Science Foundation of China (Grant No. 11571166, 11901293) and Start-up Foundation of Nanjing University of Science and Technology (No. AE89991/114). J. Yan is supported in part by the National Natural Science Foundation of China (Grant No. 11790273, 11631006)}

\begin{document}
\maketitle
\vspace{0.1in}

\tableofcontents

\begin{abstract}
  This article is devoted to a description of the dynamics of the phase flow of monotone contact Hamiltonian systems. Particular attention is paid to locating the maximal attractor (or repeller), which could be seen as the union of compact invariant sets, and investigating its topological and dynamical properties. This is based on an analysis from the viewpoint of gradient-like systems.
\end{abstract}

\section{Introduction and statement of the main results}
Let $M$ be an $n$-dimensional, closed, connected $C^{\infty}$ manifold. Equipped with the canonical symplectic form $\Om=dx\wedge dp$, the cotangent bundle $T^\ast M$ becomes a symplectic manifold. A Hamiltonian system on $(T^\ast M,\om)$ is defined as the symplectic gradient vector field of a $C^2$ function, called Hamiltonian, on $T^\ast M$. As the natural framework of classical and celestial mechanics, Hamiltonian systems received much attention and were extensively studied since the time of Newton. However, Hamiltonian systems can only be used as the mathematical model of conservative systems such as classical mechanical system or the micro-canonical ensemble in statistical mechanics. Therefore, to apply theoretical results to systems exchanging energy with an environment, one has to find a more suitable generalization of Hamiltonian systems.

\vspace{5pt}
As direct generalizations of Hamiltonian systems via characteristic theory for first order PDEs, contact Hamiltonian systems becomes a worth trying choice. Such systems are determined by the standard contact structure and a $C^2$ function on the manifold of 1-jets of functions on $M$ and have not been considered as much as their symplectic counterparts. In recent years, several applications of contact Hamiltonian dynamics have been found in equilibrium or irreversible thermodynamics, statistical physics and classical mechanics of dissipative systems. For a nice survey of such applications, we refer to \cite{Br}.

\vspace{5pt}
The mathematical study of contact Hamiltonian system begins with a typical example, the discounted systems. The dynamics of discounted system were investigated by from different aspects including quasi-periodic motions \cite{CCL1,CCL2}, Aubry-Mather sets in low dimension model \cite{Ca,LC1,LC2,LC3}, generalization of Aubry-Mather theory and weak KAM theory \cite{MS} and applications to PDE problems \cite{DFIZ,IS}. The variational theories for more general contact Hamiltonian systems with arbitrary degree of freedom were explored in the series of works \cite{SWY,WWY1,WWY2,WWY3} and \cite{CCJWY,CCWY}.

\vspace{5pt}
Among the above miscellaneous works, \cite{MS,WWY3} is of particular interests to us because variational theory is applied to understand the global dynamics of monotone contact Hamiltonian systems under Tonelli assumptions. More precisely, in \cite{WWY3}, the authors find a compact subset of the phase space containing all $\omega$-limit sets of the phase flow; while in the in \cite{MS}, the authors defines and locates the maximal attractor, i.e., the union of all compact, invariant sets, which, at least to us, is a suitable and promising concept in the investigation of the dissipative feature of such systems. Inspired by these two works, the aim of this paper is two fold, to study the maximal attractor and global dynamics of the system in more detail, and to present the results in an elementary way under more relaxed assumptions. We want to emphasis that our approach is based on an analysis of the system from the gradient-like viewpoint (constructing Lyapunov functions of the phase flow) and is independent of the variational approach mentioned before.

\vspace{5pt}
Once and for all, we choose an auxiliary complete Riemannian metric on $M$ and, with slight abuse of notation, use $\|\cdot\|_x$ to denote the norm induced on the cotangent bundle $T^{\ast}M$. Via a canonical diffeomorphism, we identify $(T^{\ast}M\times\R,\alpha)$ with the contact manifold of 1-jets of functions on $M$, where $\al=du-pdx$ denotes the standard tautological $1$-form. We consider a $C^{2}$ function $H:T^{\ast}M\times\R\rightarrow\R$, called monotone contact Hamiltonian, satisfying
\begin{itemize}
  \item [\textbf{(H1)}] $\frac{\partial^2 H}{\partial p^2}(x,p,u)$ is positive definite for every $(x,p,u)\in T^{\ast}M\times\R$,
  \item [\textbf{(H2)}] $\lim_{\|p\|_x\rightarrow\infty}H(x,p,u)\rightarrow\infty$ for every $(x,u)\in M\times\R$,
\end{itemize}
and
\begin{itemize}
  \item [\textbf{(M$_-$)}] $\frac{\partial H}{\partial u}(x,p,u)\geq\lb$ for some $\lb>0$ and every $(x,p,u)\in T^{\ast}M\times\R$,
\end{itemize}
or alternatively, \textbf{(H1)-(H2)} and
\begin{itemize}
  \item [\textbf{(M$_+$)}] $\frac{\partial H}{\partial u}(x,p,u)\leq-\lb$ for some $\lb>0$ and every $(x,p,u)\in T^{\ast}M\times\R$,
\end{itemize}
and the contact Hamiltonian vector field $X_{H}$ associated to $H$ and $\al$ by
\begin{equation}\label{def-ch}
\mathcal{L}_{X_H}\al=-\frac{\partial H}{\partial u}\al,\quad\quad\al(X_H)=-H,
\end{equation}
where $\mathcal{L}_{X_H}$ denotes the Lie derivative along the vector field $X_H$. The aim of this paper is to study the global dynamics of the phase flow $\Phi^{t}_{H}$ generated by $X_{H}$. Physically, the orbit of such a flow generalizes the motion of particles in mechanical systems with friction or simulation.

\vspace{5pt}
For any $z\in T^{\ast}M\times\R$, let $\al(z)$ (resp. $\om(z)$) denotes the $\al$-limit set of $z$ (resp. $\om$-limit set of $z$) under $\Phi^{t}_H$. It turns out that the global dynamics of $\Phi^{t}_{H}$ is closely related to a compact, invariant set called maximal attractor (or repeller), which is formally
\begin{definition}\label{mga}
A compact set $\cA\subset T^*M\times\R$ is called a global attractor (resp. repeller) of $\Phi^t_{H}$ if it is $\Phi^t_{H}$-invariant and for any $z\in T^*M\times\R, \om(z)$ (resp. $\al(z)$) $\subseteq\cA$.

\vspace{5pt}
Moreover, a global attractor (resp. repeller) $\cA$ is called maximal if it is a maximal element in the partially ordered set of all global attractors with the relation $\subseteq$.
\end{definition}

\begin{remark}
The definition of global attractor (resp. repeller) is equivalent to that on \cite[page 780]{MS}: for any neighborhood $\cO$ of $\cA$ and $z\in T^*M\times\R$, there is $T(z,\cO)>0$ (resp. $<0$) such that $\Phi^t_H(z)\in\cO$ for $t\geq T$ (resp. $t\leq T$).
\end{remark}

\begin{remark}
The notion of attractor (or repeller) is so important in the study of dynamical systems that many literatures devote to giving a widely accepted definition, for example \cite{C,Mi2,Mi3}.
\end{remark}

Although there maybe more than one global attractors for $\Phi^t_{H}$, the maximal attractor (resp. repeller) of $\Phi^t_{H}$ is unique (if exists) and equals the union of compact $\Phi^t_{H}$-invariant sets. We shall denote it by $\cA_H$.

\vspace{5pt}
Since the phase space $T^{\ast}M\times\R$ of $\Phi^t_H$ is non-compact, the existence of maximal attractor or repeller is not a trivial fact. Using tools from Aubry-Mather theory and weak KAM theory, this fact was established in \cite{MS} for the case of conformally symplectic flow, i.e. $H(x,p,u)=\lb u+h(x,p), \lb>0$, and in \cite{WWY3} for monotone contact Hamiltonians satisfying \textbf{Tonelli assumptions}.

\vspace{5pt}
Besides another proof of this fact, our first result offers some information on the topological structure of the maximal attractor (resp. repeller), namely
\begin{thmA}
Assume the contact Hamiltonian $H\in C^2(T^{\ast}M\times\R,\R)$ satisfies \textbf{(H1)-(H2)} and \textbf{(M$_-$)} (resp. \textbf{(H1)-(H2)} and \textbf{(M$_+$)}), then
\begin{itemize}
  \item [\textbf{(A1)}] $\Phi^t_H$ is forward (resp. backward) complete, i.e., $\Phi^t_H$ is well-defined for all $t\in[0,+\infty)\,\,($resp. $t\in(-\infty,0])$.

  \vspace{5pt}
  \item [\textbf{(A2)}] The maximal attractor (resp. repeller) $\cA_H$ for $\Phi^t_H$ exists and $\alpha(z)\neq\emptyset\,\,($resp. $\omega(z)\neq\emptyset)$ if and only if $z\in\cA_H$. Precisely, for every neighborhood $\cO$ of $\cA_H$ and every compact set $\cK\subset T^\ast M\times\R$, there is $T(\cK,\cO)>0\,\,($resp. $T(\cK,\cO)<0)$ such that
        \[
        \Phi^t_H(\cK)\subset\cO\quad\text{for all  }t\geq T\,\,(\text{resp. $t\leq T$}).
        \]

  \vspace{5pt}
  \item [\textbf{(A3)}] There exists a basis of neighborhoods $\{\cO_{t}\}_{t\geq0}$ of $\cA_H$ such that every $\cO_{t}$ is homotopic equivalent to $M$. In particular, $\cA_H$ is connected.
\end{itemize}
\end{thmA}

\begin{remark}
\textbf{(A2)} obviously implies Definition \ref{mga}. The statement \textbf{(A2)} means that we may choose $T(z,\cO)$ to be uniform for all $z$ in any compact set $\cK$ of $\,T^{\ast}M\times\R$.
\end{remark}

To study the dynamics of $\Phi^t_H$ on $\cA_H$, we shall focus on monotone contact Hamiltonians satisfying the additional assumptions:
\vspace{3pt}
\begin{enumerate}
  \item [\textbf{(H3)}] $\frac{\partial H}{\partial p}(x,0,u)=0$ for every $(x,u)\in M\times\R$,
\end{enumerate}
\vspace{3pt}
and, if we set
\[
\cF_{H}=\{(x_0,0,u_0)\in T^{\ast}M\times\R\,:\,\partial_{x}H(x_0,0,u_0)=0, H(x_0,0,u_0)=0\},
\]
\begin{enumerate}
  \item [\textbf{(H4)}] $x_0$ is a nondegenerate critical point of the function $x\mapsto H(x,0,u_0)$ for every $(x_0,u_0)\in\cF_{H}$.
\end{enumerate}

\begin{remark}
\textbf{(H3)} means that the convex function $p\mapsto H(x,p,u)$ attains its minimum at $0\in T^{\ast}_xM$ for every $(x,u)\in M\times\R$. According to \textbf{(H3)}, $\cF_{H}$ consists of all equilibria of $X_H$.
\end{remark}

\begin{remark}
Combining \textbf{(H4)} and the compactness of $M$, for every $u_0\in\R$, there are only finitely many isolated points $x_0\in M$ such that $(x_0,u_0)\in\cF_{H}$. Notice that contact Hamiltonians satisfying \textbf{(H4)} form an open and dense subset of the space of contact Hamiltonians satisfying \textbf{(H1)-(H3)}.
\end{remark}

We further explain the assumptions \textbf{(H3)-(H4)} by giving the following
\begin{example}
There is a natural class of contact Hamiltonians satisfying \textbf{(H3)}, namely ones satisfying the symmetry assumption:
\[
H(x,p,u)=H(x,-p,u)\quad \text{ for every }\quad (x,p,u)\in T^{\ast}M\times\R.
\]

In particular, the Hamiltonian $H:T^{\ast}M\times\R\rightarrow\R$ defined by
\begin{equation}\label{dmh}
H(x,p,u)=K(x,p)+V(x)+\lb u,\\
\end{equation}
where $K(x,p)=\frac{1}{2}\|p\|^{2}_{x}$ denotes the kinetic energy of moving particles and $V:M\times\R\rightarrow\R$ is a $C^{2}$ potential, is a monotone Hamiltonian satisfying \textbf{(H3)}. In addition, if $V$ is a Morse function, then \textbf{(H4)} is satisfied.
\end{example}

To describe the structure of $\cA_H$, we need an elementary class of invariant sets other than equilibria. For two distinct equilibria $z_0,z_1\in\cF_H$, we use $\Sigma(z_0,z_1)$ to denote the set of all $z\in T^{\ast}M\times\R$ satisfying
\[
\lim_{t\rightarrow-\infty}\Phi^t_H(z)=z_0,\quad \lim_{t\rightarrow+\infty}\Phi^t_H(z)=z_1.
\]
Let
\[
\Sigma_H=\bigcup_{z_0,z_1\in\cF_H}\,\Sigma(z_0,z_1),
\]
then $\Sigma_H$ is clearly $\Phi^{t}_{H}$-invariant. Our second main result can be summarized into the following
\begin{thmB}\label{dyn1}
Assume the contact Hamiltonian $H\in C^2(T^{\ast}M\times\R,\R)$ satisfies \textbf{(H1)-(H4)} and \textbf{(M$_-$)} (resp. \textbf{(H1)-(H4)} and \textbf{(M$_+$)}), then
\begin{itemize}
  \item [\textbf{(B1)}] The maximal attractor (resp. repeller) $\cA_H=\cF_{H}\cup\Sigma_H$. In particular, $\cA_H$ is path-connected.

  \vspace{5pt}
  \item [\textbf{(B2)}] For any $z\in T^{\ast}M\times\R$, both $\alpha(z), \omega(z)$ consist of at most one equilibrium in $\cF_H$.
  \vspace{5pt}
  \item [\textbf{(B3)}] For two equilibria $z_0=(x_0,0,u_0),z_1=(x_1,0,u_1)\in\cF_H$, assume $\Sigma(z_0,z_1)\neq\emptyset$, then $u_0<u_1$.
\end{itemize}
\end{thmB}

\vspace{5pt}
The remaining of this paper is organized as follows. In Section 2, we analysis system \eqref{ch} from the viewpoint of gradient-like system by deriving two Lyapunov functions that are crucial for the proof of Theorem A. Section 3 is devoted to the proof of Theorem A. In Section 4, we prove Theorem B and apply it to the discounted systems. The appendix contains preliminaries which maybe helpful for understanding the main body of this paper.

\section{Analysis from a gradient-like viewpoint}
An important feature of the monotone contact Hamiltonian systems is that the phase flow $\Phi^t_H$ possesses various Lyapunov functions on corresponding domains of the phase space. In this section, we shall give a detailed analysis of $\Phi^t_H$ in this direction, then use these results to derive some forward or backward invariant sets. Efforts are paid to the presentation to minimize the preknowledge.

\vspace{5pt}
\subsection{Settings and first Lyapunov function}
~\\
Let $TM$ and $T^{\ast}M$ denote the tangent and cotangent bundle of $M$. A point of $TM$ will be denoted by $(x,\dot{x})$, where $x\in M$ and $\dot{x}\in T_xM$, and a point of $T^{\ast}M$ by $(x,p)$, where $p\in T^{\ast}_{x}M$ is a linear form on $T_{x}M$. The canonical pairing between tangent and cotangent bundles are denoted by $\langle\cdot,\cdot\rangle$. We shall use either $z$ or the local coordinates $(x,p,u)$ to denote a point of $T^{\ast}M\times\R$. In this and later sections, we always assume the contact Hamiltonian $H\in C^{2}(T^{\ast}M\times\R,\R)$ satisfies \textbf{(H1)-(H2)}.

\vspace{5pt}
To begin with, we use local coordinates to express the tautological $1$-form $\al$ and the contact Hamiltonian vector field $X_H$ as $\al=du-pdx$, and
\begin{equation}\label{ch}
X_H:
\begin{cases}
\dot{x}=\frac{\partial H}{\partial p},\\
\dot{p}=-\frac{\partial H}{\partial x}-\frac{\partial H}{\partial u}\cdot p,\\
\dot{u}=\frac{\partial H}{\partial p}\cdot p-H,
\end{cases}
\end{equation}
respectively. Since $H\in C^{2}(T^{\ast}M\times\R,\R)$, \eqref{ch} shows that $X_H$ is a $C^1$ vector field on $T^{\ast}M\times\R$. Thus the fundamental theorems for ordinary differential equations states that, for every $z=(x,p,u)\in T^{\ast}M\times\R$, there is a unique integral curve of $X_H$ passing through $z$, with the maximum existence interval $(a(z),b(z))$, where $a(z)<0<b(z)$. We use
\[
\text{either}\,\,\Phi^{t}_H(z)\quad\quad\text{or}\,\, z(t)=(x(t),p(t),u(t)),\quad t\in(a(z),b(z))
\]
to denote the integral curve through $z$. Notice that $z(t)$ is $C^1$ with respect to $t$, by \eqref{ch}, one can compute as
\begin{align*}
&\dot{H}(z(t))=\langle dH,X_H\rangle\,(z(t))=\bigg[\frac{\partial H}{\partial x}\cdot\dot{x}+\frac{\partial H}{\partial p}\cdot\dot{p}+\frac{\partial H}{\partial u}\cdot\dot{u}\bigg]\,(z(t))\\
=\,\,&\bigg[\frac{\partial H}{\partial x}\cdot\frac{\partial H}{\partial p}-\frac{\partial H}{\partial p}\cdot\big(\frac{\partial H}{\partial x}+\frac{\partial H}{\partial u}\cdot p\big)+\frac{\partial H}{\partial u}\cdot\big(\frac{\partial H}{\partial p}\cdot p-H\big)\bigg]\,(z(t))\\
=\,\,&-\frac{\partial H}{\partial u}\,(z(t))\cdot H\,(z(t)),
\end{align*}
which states that
\begin{equation}\label{energy}
H(z(t))=e^{-\int^{t}_0\frac{\partial H}{\partial u}(z(s))ds}\cdot H(z),\quad t\in(a(z),b(z)).
\end{equation}
Combining \eqref{energy} and the monotonicity assumptions on $H$, one has the well-known

\begin{prop}[\textbf{First Lyapunov function}]\label{Ly-1st}
For all $z\in T^{\ast}M\times\R$, the contact Hamiltonian $H$ never changes its sign along the integral curve $\Phi^t_H(z)$. Moreover, assume $H:T^{\ast}M\times\R\rightarrow\R$ satisfies \textbf{(M$_-$)} (resp. \textbf{(M$_+$)}), then
\begin{equation}\label{Ly-1}
\begin{split}
|H(z(t))|&\leq e^{-\lb t}\cdot|H(z)|,\quad t\in[0,b(z))\\
(\text{resp.}\quad|H(z(t))|&\leq e^{\lb t}\cdot|H(z)|,\quad t\in(a(z),0]).
\end{split}
\end{equation}
\end{prop}

As an immediate consequence of the above proposition, we have
\begin{corollary}\label{pos-inv1}
Assume $H:T^{\ast}M\times\R\rightarrow\R$ satisfies \textbf{(M$_-$)} (resp. \textbf{(M$_+$)}), then for $\delta>0$, $H^{-1}(0)$ and $H^{-1}((-\infty,\delta))$ is forward (resp. backward) invariant under $\Phi^{t}_{H}$.
\end{corollary}

\begin{proof}
Assume $H(z)\in[e_{-},e_{+}]$ for some $e_{-}\leq0\leq e_{+}$. Then using Proposition \ref{Ly-1st}, we have
\begin{align*}
e_{-}\leq e^{-\lb t}e_{-}\leq H(z(t))\leq e^{-\lb t}e_{+}\leq e_{+},\quad t\in[0,b(z)),\\
(\text{resp. }e_{-}\leq e^{\lb t}e_{-}\leq H(z(t))\leq e^{\lb t}e_{+}\leq e_{+},\quad t\in(a(z),0],)
\end{align*}
which implies that $z(t)\in H^{-1}([e_{-},e_{+}])$. This completes the proof.
\end{proof}

\vspace{5pt}
\subsection{HJ equation and second Lyapunov function}
~\\
It is known \cite[Chapter 1-2]{Ar2}, by the wave-particle duality, that \eqref{ch} is the characteristic system for some Hamilton-Jacobi PDE associated to $H$ so that the dynamics of $\Phi^t_H$ is closely related to the solution of the corresponding equation.

\vspace{5pt}
Assume $H:T^{\ast}M\times\R\rightarrow\R$ satisfies \textbf{(M$_-$)} (resp. \textbf{(M$_+$)}), we consider the Hamilton-Jacobi equation
\begin{align*}
H(x,\partial_x u,u)&=0,\quad x\in M \tag{HJ$_-$}\\
(\text{resp.}\quad H(x,-\partial_x u,-u)&=0,\quad x\in M \tag{HJ$_+$}).
\end{align*}
Notice that if $H(x,p,u)$ satisfies \textbf{(H1)-(H2)} and \textbf{(M$_+$)}, then $H(x,-p,-u)$ satisfies \textbf{(H1)-(H2)} and \textbf{(M$_-$)}. From the theory of Hamilton-Jacobi equations, we know that (HJ$_-$) (resp. (HJ$_+$)) admits a unique solution $u_-$ (resp. $-u_+)\in C(M,\R)$, which is in general not $C^{1}$ and should be understood in the viscosity sense. By the theory of viscosity solutions \cite{Ishii}, $u_\pm$ is Lipschitz continuous on $M$. Rademacher Theorem implies $du_\pm$ exists for almost every point on $M$. Thus for every $x\in M$, the set
\[
D^{\ast}u_\pm(x)=\{p\in T^\ast_x M: \exists \{x_k\}\subseteq M\setminus\{x\}, x=\lim_{k\rightarrow\infty}x_k, p=\lim_{k\rightarrow\infty}du_\pm(x_k)\}
\]
is non-empty and compact. A more detailed analysis \cite[Theorem 3.2.1,\,\, Proposition 5.3.1,\,\, Theorem 5.3.6]{CS} shows that

\begin{prop}\label{vis}
$u_-$ (resp. $u_+$) $:M\rightarrow\R$ is locally semiconcave (resp. semiconvex) and for every $x\in M$ and $p\in D^\ast u_\pm(x)$,
\[
H(x,p,u_\pm(x))=0.
\]
In particular, for every $\dot{x}\in T_xM$, the directional derivative
\[
\partial u_\pm(x,\dot{x})=\lim_{h\rightarrow0^+}\frac{u_\pm(x+h\dot{x})-u_\pm(x)}{h}
\]
exists and
\begin{equation}\label{dd}
\begin{split}
\partial u_-(x,\dot{x})=\min_{p\in D^{\ast}u_-(x)}\langle p,\dot{x}\rangle,\\
\partial u_+(x,\dot{x})=\max_{p\in D^{\ast}u_+(x)}\langle p,\dot{x}\rangle.
\end{split}
\end{equation}
\end{prop}

\begin{remark}
The function $u_-$ (resp. $u_+$) coincides with the backward (resp. forward) weak KAM solution for (HJ$_-$) defined in \cite{WWY3}.
\end{remark}

The next lemma is well-known in convex analysis, it is used to derive the second Lyapunov function for $\Phi^t_H$. The proof is omitted, interested reader may refer to \cite[Theorem 23.5]{R}.
\begin{lemma}\label{t1}
Assume $H: T^{\ast}M\times\R\rightarrow\R$ is a $C^1$ function satisfying \textbf{(H1)}, then
\[
\langle p,\dot{x}\rangle-H(x,p,u)=\sup_{p^{\prime}\in T^\ast_xM}\big\{\langle p^{\prime},\dot{x}\rangle-H(x,p^{\prime},u)\big\}
\]
if and only if $\dot{x}=\frac{\partial H}{\partial p}(x,p,u)$.
\end{lemma}

\begin{remark}
The coercive assumption \textbf{(H2)} can not ensure that for every $(x,\dot{x})\in TM$,
\[
\sup_{p^{\prime}\in T^\ast_xM}\big\{\langle p^{\prime},\dot{x}\rangle-H(x,p^{\prime},u)\big\}
\]
is attained by some $p\in T^\ast_xM$.
\end{remark}

For $z=(x,p,u)$, assume $H:T^{\ast}M\times\R\rightarrow\R$ satisfies \textbf{(M$_-$)} (resp. \textbf{(M$_+$)}), let us define
\begin{equation}\label{def-Ly2}
F(z)=u_-(x)-u\quad (\text{resp. }F(z)=u-u_+(x)),
\end{equation}
then the following theorem shows that, besides $|H|, F$ serves as another Lyapunov function on $F^{-1}([0,+\infty))$.

\begin{theorem}[\textbf{Second Lyapunov function}]\label{Ly-2nd}
For every $z=(x,p,u)$ such that $F(z)\geq0$,
\begin{equation}\label{Ly-2}
\cL_{X_H}F(z)\leq -\lb F(z)\quad(\text{resp.}\quad\cL_{-X_H}F(z)\leq -\lb F(z)),
\end{equation}
with the left hand vanishes if and only if $u=u_-(x), p\in D^\ast u_-(x)$ (resp. $u=u_+(x), p\in D^\ast u_+(x)$) and for any $p^{\prime}\in D^\ast u_-(x)$ (resp. $p^{\prime}\in D^\ast u_+(x)$),
\[
\langle p-p^{\prime},\frac{\partial H}{\partial p}(x,p,u)\rangle\leq0.
\]
In particular, $F$ is monotone decreasing along forward (resp. backward) $\Phi^{t}_H$-orbit segments in $F^{-1}([0,+\infty))$.
\end{theorem}

\begin{proof}
Notice that the existence of direction derivative follows from Proposition \ref{vis}. In this proof, we set $\dot{x}=\langle dx,X_H\rangle, \dot{u}=\langle du,X_H\rangle$.

\vspace{5pt}
Assume $H$ satisfies \textbf{(M$_-$)} and $F(z)\geq0$. For any $p^{\prime}\in D^\ast u_-(x)$, we use \eqref{ch} to compute
\begin{align*}
&\cL_{X_H}F(z)=\partial u_-(x,\dot{x})-\dot{u}\\
\leq&\,\langle p^{\prime},\dot{x}\rangle-\dot{u}=\langle p^{\prime},\dot{x}\rangle-\big[\langle p,\dot{x}\rangle-H(x,p,u)\big]\\
\leq&\,\langle p^{\prime},\dot{x}\rangle-\big[\langle p^{\prime},\dot{x}\rangle-H(x,p^{\prime},u)\big]\\
=&\,H(x, p^{\prime}, u)-H(x, p^{\prime}, u_-(x))\\
\leq&\,\lb (u-u_-(x))=-\lb F(z),
\end{align*}
The third equality follows from Proposition \ref{vis}. The first inequality follows from \eqref{dd}, the second one uses Lemma \ref{t1} and the third one is due to \textbf{(M$_-$)}.

\vspace{5pt}
Assume $H$ satisfies \textbf{(M$_+$)} and $F(z)\geq0$. We use \eqref{ch} to compute
\begin{align*}
&\cL_{-X_H}F(z)=-\dot{u}-\partial u_+(x,-\dot{x})\\
=&\,-\dot{u}+\langle\bar{p},\dot{x}\rangle=\langle\bar{p},\dot{x}\rangle-\big[\langle p,\dot{x}\rangle-H(x,p,u)\big]\\
\leq&\,\langle\bar{p},\dot{x}\rangle-\big[\langle\bar{p},\dot{x}\rangle-H(x,\bar{p},u)\big]\\
=&\,H(x, \bar{p}, u)-H(x, \bar{p}, u_+(x))\\
\leq&\,\lb (u_+(x)-u)=-\lb F(z),
\end{align*}
where $\bar{p}\in D^\ast u_+(x)$ satisfies $\partial u_+(x,-\dot{x})=\langle\bar{p},-\dot{x}\rangle$. The fourth equality follow from Proposition \ref{vis}. The first inequality uses Lemma \ref{t1}, the second one is due to \textbf{(M$_+$)}.

\vspace{5pt}
Note that if the left hand of \eqref{Ly-2} vanishes, then every inequality in the above computation occurs to be an equality. Let $\bar{p}\in D^\ast u_\pm(x)$ satisfies
\[
\langle\bar{p},\dot{x}\rangle=\min_{p^\prime\in D^{\ast}u_\pm(x)}\langle p^\prime,\dot{x}\rangle.
\]
The inequality using Lemma \ref{t1} gives $p=\bar{p}$ and the last inequality becomes an equality if and only if $u=u_-(x)$ (resp. $u=u_+(x)$).
\end{proof}

From the above theorem, it is natural to define
\begin{align*}
\cU=&\,\{z\in T^{\ast}M\times\R:F(z)\leq 0\}\\
=&\,\{(x,p,u)\in T^{\ast}M\times\R:u\geq u_-(x)\}\\
(\text{resp. }=&\,\{(x,p,u)\in T^{\ast}M\times\R:u\leq u_+(x)\}),
\end{align*}
and for $\delta>0$,
\begin{align*}
\cU_{\delta}=&\,\{z\in T^{\ast}M\times\R:F(z)<\delta\}\\
=&\,\{(x,p,u)\in T^{\ast}M\times\R:u>u_-(x)-\delta\}\\
(\text{resp. }=&\,\{(x,p,u)\in T^{\ast}M\times\R:u<u_+(x)+\delta\}).
\end{align*}
The following corollary is a direct consequence of Theorem \ref{Ly-2nd}.
\begin{corollary}\label{pos-inv2}
Assume $H:T^{\ast}M\times\R\rightarrow\R$ satisfies \textbf{(M$_-$)} (resp. \textbf{(M$_+$)}), $\cU$ and $\cU_\delta,\,\,\delta>0$ are forward (resp. backward) invariant under $\Phi^{t}_{H}$.
\end{corollary}

\begin{proof}
We shall give the proof in case that $H$ satisfies \textbf{(M$_-$)}, the other part is completely similar. For every $z\in\cU$, set $f(t)=F\circ z(t), t\in[0, b(z))$. We argue by contradiction: assume that there is $z\in\cU$ and $T>0$, such that
\[
f(0)\leq0,\quad f(T)>0.
\]
By the continuity of the function $f(t)$, there is $\tau\in[0,T]$ such that
\[
f(\tau)=0,\quad f(t)\geq0\quad\text{ for}\,\,t\in[\tau,T].
\]
This contradicts with Theorem \ref{Ly-2nd} since $\tau<T$ and
\[
F\circ z(\tau)=f(\tau)=0<f(T)=F\circ z(T).
\]

Notice that if $z\in \cU_\delta$ and $z(T)\notin\cU_\delta$, then there is $\tau>0$ such that
\[
0<f(\tau)<\delta,\quad f(T)\geq\delta.
\]
To complete the proof, one use the same argument as above.
\end{proof}

\section{Proof of Theorem A}
For a monotone contact Hamiltonian $H$, it is found that most orbits of $\Phi^t_H$ goes forward (resp. backward) to some compact, $\Phi^t_H$-invariant set called maximal attractor (resp. repeller). This section is devoted to the investigation of maximal attractor (resp. repeller), it is divided into three sections corresponding to the conclusions in Theorem A.

\vspace{5pt}
In this section, $U_0\geq0$ denotes the $C^0$-norm of $u_-$ (resp. $u_+$) $\in C(M,\R)$ associated to the contact Hamiltonian $H$ satisfying \textbf{(M$_-$)} (resp. \textbf{(M$_+$)}).

\vspace{5pt}
\subsection{Forward or backward completeness}
~\\
To show the forward (resp. backward) completeness of the flow, the first task is to build compact forward (resp. backward) invariant sets. This is due to the following simple observation:

\begin{lemma}\label{t2}
Assume $H:T^{\ast}M\times\R\rightarrow\R$ satisfies \textbf{(M$_-$)} (resp. \textbf{(M$_+$)}), then for every $e,U\in\R$,
\begin{align*}
\cY(e,U)&:=\{(x,p,u)\in T^\ast M\times\R: H(x,p,u)\leq e, u\geq U\}\\
(\text{resp. }&:=\{(x,p,u)\in T^\ast M\times\R: H(x,p,u)\leq e, u\leq U\})
\end{align*}
is compact.
\end{lemma}

\begin{proof}
From the definition, $\cY(e,U)$ is closed and that by \textbf{(H2$^\prime$)} (see Appendix), for all $(x,p,u)\in\cY(e,U), \|p\|_x\leq P(e,U)$. Now the assumption \textbf{(M$_{\pm}$)} implies
\[
|u|\leq\frac{1}{\lb}|H(x,p,u)-H(x,p,0)|\leq\frac{1}{\lb}\bigg(e_+ +\max_{\|p\|_x\leq P}|H(x,p,0)|\bigg)\,.
\]
Thus $\cY(e,U)$ is also bounded, this completes the proof.
\end{proof}

\vspace{5pt}
We define
\[
\cY=H^{-1}(0)\cap\cU,
\]
and for $\delta>0$,
\begin{equation}\label{def-Y}
\cY_{\delta}=H^{-1}((-\infty,\delta))\cap\cU_{\delta}.
\end{equation}
Denote by $\overline{\cY_{\delta}}$ the closure of $\cY_{\delta}$, it follows directly from Lemma \ref{t2} that
\begin{theorem}\label{pos-inv}
Assume $H:T^{\ast}M\times\R\rightarrow\R$ satisfies \textbf{(M$_-$)} (resp. \textbf{(M$_+$)}), $\cY$ and $\overline{\cY_{\delta}}$ are compact and forward (resp. backward) invariant under $\Phi^{t}_{H}$.
\end{theorem}

\begin{proof}
By definition, \textbf{(H2)} and \textbf{(M$_{-}$)} (resp. \textbf{(M$_{+}$)}),
\[
\overline{\cY_{\delta}}=H^{-1}((-\infty,\delta])\cap F^{-1}((-\infty,\delta]),\quad\text{and}\quad \cY\subset\overline{\cY_{\delta}}
\]
are closed sets. It is easy to see that
\[
\overline{\cY_{\delta}}\subseteq\cY(\delta, -U_0-\delta)\quad\text{resp.}\quad\overline{\cY_{\delta}}\subseteq\cY(\delta, U_0+\delta).
\]
By Lemma \ref{t2}, being a closed subset of a compact set, $\cY_{\delta}$ and $\cY$ must be compact. The forward (resp. backward) invariance of $\cY_{\delta}$ and $\cY$ follows directly from Corollary \ref{pos-inv1} and \ref{pos-inv2}.
\end{proof}

\vspace{5pt}
An important corollary of the above proposition is

\vspace{5pt}
\textit{Proof of }\textbf{(A1)}:
Assume $H:T^{\ast}M\times\R\rightarrow\R$ satisfies \textbf{(M$_-$)} (resp. \textbf{(M$_+$)}), by Theorem \ref{Ly-2nd} and Corollary \ref{pos-inv2}, for every $z\in T^{\ast}M\times\R$, we have the following dichotomy:
\begin{itemize}
  \item there is $T\in[0,b(z))$ (resp. $T\in(a(z),0]$) such that $\Phi^t_H(z)\in\cU$ for $t\in[T,b(z))\,\,(\text{resp.}\,\,\,t\in(a(z),T]$.

  \vspace{5pt}
  \item $\{\Phi^t_H(z)\}_{t\in[0,b(z))}\subseteq\cU^c$, then it follows that
        \begin{equation}\label{ieq:1}
        \begin{split}
        0\leq F(\Phi^t_H(z))\leq e^{-\lb t}F(z)\leq F(z),\quad t\in[0,b(z))\\
        (\text{resp.}\quad 0\leq F(\Phi^t_H(z))\leq e^{\lb t}F(z)\leq F(z),\quad t\in(a(z),0]).
        \end{split}
        \end{equation}
\end{itemize}
Set $e_z=|H(z)|, U_z=-U_0-F(z)$ (resp. $U_z=U_0+F(z)$), combining Proposition \ref{Ly-1st} and the above, for every $z\in T^{\ast}M\times\R$,
\[
\Phi^t_H(z)\in\cY(e_z, U_z),\quad t\in[0,b(z))\,\,(\text{resp.}\,\,t\in(a(z),0]).
\]
Now we use Proposition \ref{t2} and Corollary \ref{ext} to see that $b(z)=\infty$ (resp. $a(z)=-\infty$) and the conclusion follows.\qed

\vspace{5pt}
Together with \eqref{Ly-1} and \eqref{ieq:1}, we obtain the by-product that for every $z\in T^{\ast}M\times\R$,
\begin{equation}\label{om-al}
\omega(z)\,\,(\text{resp.}\,\,\al(z))\in H^{-1}(0)\cap\cU=\cY
\end{equation}

\begin{remark}
It is equivalent to say, if $H$ satisfies \textbf{(M$_-$)} (resp. \textbf{(M$_+$)}), then $\Phi^t_H:T^{\ast}M\times\R\rightarrow T^{\ast}M\times\R$ is well-defined for any $t\geq0$ (resp. $t\leq0$).
\end{remark}

\vspace{5pt}
\subsection{Existence and location}
~\\
As the second step, we prove the existence of maximal attractor (resp. repeller) by locating the maximal compact, $\Phi^t_H$-invariant set. This fact serves as a stronger version of \eqref{om-al}.
\begin{theorem}\label{cpt-inv}
Any compact $\Phi^{t}_{H}$-invariant set is contained in $\cY$.
\end{theorem}

\begin{proof}
Assume there is $z=(x,p,u)\in T^{\ast}M\times\R$ and a compact set $\cK\subset T^{\ast}M\times\R$ such that $\Phi^{t}_{H}(z)=z(t)=(x(t),p(t),u(t))\in\cK$ for all $t\in(a(z),b(z))$. We shall prove $z\in\cY$.

\vspace{5pt}
By Lemma \ref{exist}, it is easy to see that $(a(z),b(z))=\R$.

\vspace{5pt}
There is $e(\cK)>0$ such that $|H(z^{\prime})|\leq e(\cK)$ for any $z^{\prime}\in\cK$. It follows from Proposition \ref{Ly-1st} that for any $T>0$,
\begin{align*}
|H(z)|\leq&\, e^{-\lb T}|H(z(-T))|\leq e^{-\lb T}e(\cK),
\end{align*}
we send $T$ goes to infinity to find $z\in H^{-1}(0)$.

\vspace{5pt}
There is $U(\cK)>0$ such that $|F(z^{\prime})|\leq U(\cK)$ for any $z^{\prime}\in\cK$. Similarly, one uses Theorem \ref{Ly-2nd} to deduce that for any $T>0$
\[
F(z)\leq e^{-\lb T}F(z(-T))\leq e^{-\lb T}U(\cK),
\]
Sending $T$ goes to infinity to find $F(z)\leq0$ and $z\in\cU\cap H^{-1}(0)=\cY$.
\end{proof}

As the consequences of \textbf{(A1)} and Theorem \ref{cpt-inv} implies
\begin{theorem}\label{exist-Atr}
The maximal global attractor $\cA_H$ for $\Phi^t_H$ exists and
\[
\cA_H\subset\cY.
\]
Moreover, $\alpha(z)\neq\emptyset\,\,($resp. $\omega(z)\neq\emptyset)$ if and only if $z\in\cA_H$.
\end{theorem}

\begin{proof}
Let $\cA_H$ be the closure of the union of all compact $\Phi^t_H$-invariant sets. By Theorem \ref{cpt-inv}, $\cA_H$ is a closed subset of $\cY$, thus is compact. It is easy to verify, by continuity of $\Phi^t_H$, that for all $z\in\cA_H,\,\,(a(z),b(z))=\R$ and $\cA_H$ itself is $\Phi^t_H$-invariant.

\vspace{5pt}
Assume $H$ satisfies \textbf{(M$_-$)} (resp. \textbf{(M$_+$)}), for every $z\in T^\ast M\times\R$, by \eqref{om-al}, $\omega(z)$ (resp. $\al(z)$) is a compact $\Phi^t_H$-invariant set in $\cY$, thus a subset of $\cA_H$. So any neighborhood $\cO$ of $\cA_H$ is also a neighborhood of $\omega(z)$ (resp. $\al(z)$) and, by the definition of $\omega$-limit set (resp. $\al$-limit set), there is $T(z,\cO)>0$ (resp. $T(z,\cO)<0$) such that $\Phi^t_H(z)\in\cO$ for all $t\geq T$ (resp. $t\leq T$). Thus by Definition \ref{mga}, $\cA_H$ is a global attractor (resp. repeller) of $\Phi^t_H$ and its maximality is implied by the definition.

\vspace{5pt}
Assume $H$ satisfies \textbf{(M$_-$)} (resp. \textbf{(M$_+$)}) and $\al(z)\neq\emptyset$ (resp. $\om(z)\neq\emptyset$). From Lemma \ref{exist}, we conclude that $(a(z),b(z))=\R$. Using equation \eqref{energy}, it is clear that $z\in H^{-1}(0)$. Thus, by Theorem \ref{Ly-2nd},
\[
\lim_{t\rightarrow-\infty}F(z(t)) (\text{resp.}\,\,\lim_{t\rightarrow+\infty}F(z(t)))=\sup_{t\in\R}F(z(t))\quad\text{exists.}
\]
Now Lemma \ref{t2} ensures that $\overline{\text{Im}(z(t))}$ is a compact $\Phi^t_H$-invariant set and is contained in $\cA_H$. This completes the proof.
\end{proof}

\begin{remark}\label{graph}
Let $\pi:T^\ast M\times\R\rightarrow T^\ast M$ be the projection forgetting $u$. Since $\cA_H\subset\cY\subset H^{-1}(0)$, $\pi|_{\cA_H}$ is a homeomorphism onto its image.
\end{remark}

\vspace{5pt}
\subsection{Attractiveness and topological properties}
~\\
The aim of this section is to prove the remaining conclusion of \textbf{(A2)} and \textbf{(A3)}, thus complete the proof of Theorem A.

\vspace{5pt}
Assume $H$ satisfies \textbf{(M$_-$)} (resp. \textbf{(M$_+$)}), let us define
\begin{equation}\label{def-attractor}
\hat{\cA}_H=\bigcap_{t\geq0}\Phi^{t}_{H}\,(\overline{\cY_{\delta}})\quad(\text{resp.}\,\,\hat{\cA}_H=\bigcap_{t\leq0}\Phi^{t}_{H}\,(\overline{\cY_{\delta}})),
\end{equation}
then $\hat{\cA}_H$ is compact and by Theorem \ref{pos-inv}, for every $T\geq\tau>0$ (resp. $T\leq\tau<0$),
\begin{equation}\label{filt}
\begin{split}
\hat{\cA}_H\subseteq\bigcap_{t\in[0,T]}\Phi^{t}_{H}\,(\overline{\cY_{\delta}})=\Phi^{T}_{H}(\overline{\cY_{\delta}})\subseteq\Phi^{\tau}_{H}(\overline{\cY_{\delta}})\\
(\text{resp.   }\hat{\cA}_H\subseteq\bigcap_{t\in[T,0]}\Phi^{t}_{H}\,(\overline{\cY_{\delta}})=\Phi^{T}_{H}(\overline{\cY_{\delta}})\subseteq\Phi^{\tau}_{H}(\overline{\cY_{\delta}})).
\end{split}
\end{equation}
This directly implies that
\begin{theorem}\label{pro-attractor2}
$\hat{\cA}_H$ is the maximal compact $\Phi^t_H$-invariant set. In particular,
\[
\cA_H=\hat{\cA}_H.
\]
\end{theorem}

\begin{proof}
We shall only consider the case that $H$ satisfies \textbf{(M$_-$)}. For the other case, we only need to replace all $\geq$ by $\leq$ in the following argument.

\vspace{5pt}
\textit{$\hat{\cA}_H$ is $\Phi^t_H$-invariant:} for every $\tau\geq0$, by Theorem \ref{pos-inv},
\[
\Phi^\tau_H(\hat{\cA}_H)=\bigcap_{t\geq0}\Phi^{\tau+t}_{H}\,(\overline{\cY_{\delta}})=\bigcap_{t\geq0}\Phi^{t}_{H}\bigg(\Phi^{\tau}_{H}(\overline{\cY_{\delta}})\bigg)\subseteq\bigcap_{t\geq0}\Phi^{t}_{H}(\overline{\cY_{\delta}})=\hat{\cA}_H,
\]
where the first equality holds since $\Phi^\tau_H$ is injective. By the definition of $\hat{\cA}_H$,
\[
\hat{\cA}_H\subseteq\bigcap_{t\geq\tau}\Phi^{t}_{H}\,(\overline{\cY_{\delta}})=\bigcap_{t\geq0}\Phi^{\tau+t}_{H}\,(\overline{\cY_{\delta}})=\Phi^\tau_H\bigg(\bigcap_{t\geq0}\Phi^{t}_{H}\,(\overline{\cY_{\delta}})\bigg)=\Phi^\tau_H(\hat{\cA}_H).
\]
Therefore we obtain
\begin{equation}\label{inv}
\Phi^{\tau}_{H}(\hat{\cA}_H)=\hat{\cA}_H,\quad\text{for}\quad\tau\geq0.
\end{equation}
Since $\hat{\cA}_H$ is compact, Lemma \ref{exist} and \eqref{inv} show that for every $z\in\hat{\cA}_H$,
\[
(a(z), b(z))=\R\quad\text{and}\quad\Phi^{\tau}_{H}(\hat{\cA}_H)=\hat{\cA}_H,\quad\text{for}\quad\tau\in\R.
\]

\vspace{5pt}
\textit{$\hat{\cA}_H$ is maximal:} assume $\cK$ is a compact $\Phi^t_H$-invariant set. From Theorem \ref{cpt-inv} and definition of $\overline{\cY_{\delta}}$,
\[
\cK\subseteq\cY\subseteq\overline{\cY_{\delta}}.
\]
The above relation and $\Phi^t_H$-invariance of $\cK$ give
\[
\cK=\bigcap_{t\geq0}\Phi^{t}_{H}\,(\cK)\subseteq\bigcap_{t\geq0}\Phi^{t}_{H}\,(\overline{\cY_{\delta}})=\hat{\cA}_H.
\]
Combining the above discussions, Theorem \ref{exist-Atr} and the compactness of $\hat{\cA}_H$, we have $\cA_H=\hat{\cA}_H$.
\end{proof}

\begin{remark}\label{flexible-Atr}
Assume $H$ satisfies \textbf{(M$_-$)} (resp. \textbf{(M$_+$)}), by repeating the proof above, it is clear that
\[
\cA_H=\bigcap_{t\geq0}\Phi^{t}_{H}\,(\cK)\quad(\text{resp.}\,\,\cA_H=\bigcap_{t\leq0}\Phi^{t}_{H}\,(\cK)),
\]
$\cK$ is any compact, forward (resp. backward) invariant set containing $\cY$.
\end{remark}

The following lemma shows that $\Phi^{T}_{H}(\overline{\cY_{\delta}})$ is a good approximation of $\cA_H$ (in the sense of topology) when $T>0$ is large enough.

\begin{lemma}\label{ap-Atr}
For every open neighborhood $\cO$ of $\cA_H$, there is $T(\overline{\cY_{\delta}}, \cO)>0$ such that $\Phi^{T}_{H}(\overline{\cY_{\delta}})\subseteq\cO$.
\end{lemma}

\begin{proof}
We assume, contrary to the conclusion, that there exist $T_n>0$ and $z_n\in\Phi^{T_n}_{H}(\overline{\cY_{\delta}})$ with
\begin{equation}\label{ctr1}
\lim_{n\rightarrow\infty}T_n=\infty,\quad z_{n}\notin\cO.
\end{equation}
Since $z_n\in\overline{\cY_{\delta}}\setminus\cO$, which is clearly compact, then
\begin{equation}\label{ctr2}
z_n\rightarrow z^\ast\in\overline{\cY_{\delta}}\setminus\cO.
\end{equation}
Now by \eqref{ctr1}, for every $t\geq0$, there is $N\in\N$ such that if $n\geq N$, then $T_n\geq t$ and $z_n\in\Phi^{T_n}_{H}(\overline{\cY_{\delta}})\subseteq\Phi^{t}_{H}(\overline{\cY_{\delta}})$. It follows from compactness of $\Phi^{t}_{H}(\overline{\cY_{\delta}})$ that $z^\ast\in\Phi^{t}_{H}(\overline{\cY_{\delta}})$ for every $t\geq0$, then
\[
z^\ast\in\bigcap_{t\geq0}\Phi^{t}_{H}\,(\overline{\cY_{\delta}})=\cA_H,
\]
which contradicts \eqref{ctr2}.
\end{proof}

Now we are ready to complete the

\vspace{5pt}
\textit{Proof of} \textbf{(A2)}:
The existence of maximal attractor (resp. repeller) is settled by Theorem \ref{exist-Atr}. We turn to the second part of the conclusion.

\vspace{5pt}
Fixing $\delta>0$, for any compact set $\cK\subset T^\ast M\times\R$, we define
\[
T_1(\cK):=\max\{-\frac{1}{\lb}\ln(\frac{\delta}{\max_{z\in\cK}|H|(z)}), -\frac{1}{\lb}\ln(\frac{\delta}{\max_{z\in\cK}F(z)})\},
\]
notice that $T_1\leq0$ if and only if $\cK\subseteq\cY_\delta$. By Proposition \ref{Ly-1st} and \ref{Ly-2nd},
\begin{equation}\label{r1}
\Phi^{t}_{H}\,(\cK)\subset\cY_\delta,\quad\text{for}\quad t>T_1.
\end{equation}
We use Lemma \ref{ap-Atr} to obtain $T_2(\cO,\overline{\cY_{\delta}})>0$ such that
\begin{equation}\label{r2}
\Phi^{t}_{H}\,(\cY_\delta)\subset\cO,\quad\text{for}\quad t>T_2.
\end{equation}
By taking $T=T_1+T_2+1$ and using \eqref{r1}, \eqref{r2} above, the conclusion follows.\qed

\vspace{5pt}
To verify the conclusions of \textbf{(A3)}, it is necessary to show the
\begin{lemma}\label{connect}
$\cY_{\delta}$ is homotopic equivalent to $M$. In particular, $\cY_{\delta}$ is path-connected.
\end{lemma}

\begin{proof}
For every $(x,u)\in M\times\R$, by \textbf{(H1)-(H2)}, there is a $C^1$ map $(x,u)\mapsto P_\ast(x,u)\in T^\ast_xM$ satisfying
\[
\frac{\partial H}{\partial p}(x, P_\ast(x,u),u)=0.
\]

Assume $H$ satisfies \textbf{(M$_-$)} (resp. \textbf{(M$_+$)}), set
\[
\Sigma_{\pm}:=\{(x,P_{\ast}(x,u_{\pm}(x)),u_{\pm}(x)):x\in M\}\subset T^\ast M\times\R,
\]
then it is clear that
\begin{itemize}
  \item $H|_{\Sigma_{\pm}}\leq0$, thus $\Sigma_{\pm}\subset\cY_{\delta}$,
  \item $\Sigma_{\pm}$ is homeomorphic to $M$.
\end{itemize}

Now for $z=(x,p,u), t\in[0,1]$, define
\[
U_{\pm}(x,t)=(1-t)u+tu_{\pm}(x)
\]
and continuous maps
\begin{align*}
G_1(z,t)&\,=(x,(1-t)p+tP_\ast(x,u),u),\\
G_{2,\pm}(z,t)&\,=(x,P_{\ast}(x,U_{\pm}(x,t)),U_{\pm}(x,t)),\quad(z,t)\in\cY_{\delta}\times[0,1].
\end{align*}

\vspace{5pt}
\textit{Claim}: $G_1,\,G_{2,\pm}$ maps $\cY_{\delta}\times[0,1]$ into $\cY_{\delta}$.

\vspace{5pt}
\textit{Proof of the claim}: For any $z=(x,p,u)\in\cY_{\delta}$, since, by \textbf{(H1)},
\[
H\circ G_1(z,t)=H(x,(1-t)p+tP_\ast(x,u),u)\leq H(x,p,u)\leq\delta
\]
and $u$ is unchanged under $G_1$, we have Im$(G_1)\subset\cY_{\delta}$.

\vspace{5pt}
For the map $G_{2,\pm}$, first notice that
\[
U_{-}(x,t)>u_-(x)-\delta\quad (\text{resp.}\,\, U_{+}(x,t)<u_+(x)+\delta),
\]
thus Im$(G_{2,\pm})\subset\cU_{\delta}$. Since $H$ is monotone in $u$, for any $(x,p)\in T^\ast M$,
\begin{align*}
\text{either}\quad H(x,p,U_{\pm}(x,t))&\,\leq H(x,p,u),\\
\text{or}\quad H(x,p,U_{\pm}(x,t))&\,\leq H(x,p,u_{\pm}(x)).
\end{align*}
Correspondingly, by the definition of $P_{\ast}$, we have
\begin{align*}
\text{either}\quad H\circ G_{2,\pm}(z,t)&\,\leq H(x,p,U_{\pm}(x,t))\leq H(x,p,u)\leq\delta,\\
\text{or}\quad H\circ G_{2,\pm}(z,t)&\,\leq H(x,P(x,u_{\pm}(x)),U_{\pm}(x,t))\\
&\,\leq H(x,P(x,u_{\pm}(x)),u_{\pm}(x))\leq0.
\end{align*}
Hence, Im$(G_{2,\pm})\subset H^{-1}((-\infty,\delta])$. This completes the proof of claim.

\vspace{5pt}
Besides, it is easy to see that $G_1(\cdot,1)=G_2(\cdot,0)$, thus we construct $G_{\pm}:\cY_{\delta}\times[0,1]\rightarrow\cY_{\delta}$ by the usual concatenation
\[
G_{\pm}(z,t)=
  \begin{cases}
    G_1(z,2t),\quad t\in[0,\frac{1}{2}]; \\
    G_{2,-}(z,2t-1)\,\,(\text{resp.}\,\,G_{2,+}(z,2t-1)),\quad t\in[\frac{1}{2},1].
  \end{cases}
\]
Then we have $G_{-}$ (resp. $G_+$) is continuous and
\begin{itemize}
  \item $G(z,1)\in\Sigma_-$ (resp. $G(z,1)\in\Sigma_+$) for any $z\in\cY_{\delta}$,
  \item $G(\cdot,t)=id_{\Sigma_-}$ (resp. $G(\cdot,t)=id_{\Sigma_+}$) for all $t\in[0,1]$,
\end{itemize}
thus $G_{-}$ (resp. $G_+$) is a strong deformation retraction from $\cY_{\delta}$ to $\Sigma_-$ (resp. $\Sigma_+$). This finishes the proof.
\end{proof}

We use Lemma \ref{connect} and Lemma \ref{ap-Atr} to give

\vspace{5pt}
\textit{Proof of }\textbf{(A3)}: Assume $H$ satisfies \textbf{(M$_-$)} (resp. \textbf{(M$_+$)}). Fix $\delta>0$, we define for $t\geq0$,
\[
\cO_t=\Phi^t_H(\cY_{\delta})\quad(\text{resp.}\,\,\cO_{t}=\Phi^{-t}_H(\cY_{\delta})).
\]
Since $\Phi^t_H,t\geq0$ (resp. $t\leq0$) is a diffeomorphism, we have $\cO_t$ is homotopic equivalent to $M$. Now Lemma \ref{ap-Atr} shows that $\{\cO_t\}_{t\geq0}$ is a basis of neighborhoods of $\cA_H$.

\vspace{5pt}
To show $\cA_H$ is connected. We argue by contradiction: assume there are two disjoint compact sets $\cK, \cK^{\prime}$ such that $\cA_H=\cK\amalg\cK^{\prime}$, where $\amalg$ denotes the disjoint union. Choosing open neighborhoods $\cK\subset\cO,\cK^{\prime}\subset\cO^{\prime}$ such that
\begin{equation}\label{ctr3}
\cO\cap\cO^{\prime}=\emptyset,
\end{equation}
and $\cO\amalg\cO^{\prime}$ is an open neighborhood of $\cA_H$.

\vspace{5pt}
By Lemma \ref{ap-Atr}, there is $T>0$ such that $\Phi^{T}_{H}\,(\cY_{\delta})\subseteq\cO\amalg\cO^{\prime}$. By Lemma \ref{connect}, $\Phi^{T}_{H}\,(\cY_{\delta})$ is path-connected, thus $z\in\cK, z^{\prime}\in\cK^{\prime}$ is connected by a path in $\Phi^{T}_{H}\,(\cY_{\delta})\subseteq\cO\amalg\cO^{\prime}$. This contradicts \eqref{ctr3}.\qed

\section{Proof of Theorem B}
In this section, we shall assume that $H:T^\ast M\times\R\rightarrow\R$ satisfies the additional assumptions \textbf{(H3)-(H4)} and then present a proof of Theorem B. The crucial tool is a last Lyapunov function on $\cY$, guaranteed by the strict convexity of $H$.

\vspace{10pt}
\subsection{Third Lyapunov function}
We need the following auxiliary lemma on strictly convex functions.
\begin{lemma}\label{conv}
Let $h:\R^n\rightarrow\R$ be a $C^2$ strictly convex function, i.e., the Hessian $d^2h(p)$ is positive definite everywhere. If $dh\,(0)=0$, then
\[
dh\,(p)\cdot p\geq0\quad\text{and}\quad dh(p)\cdot p=0\quad\text{if and only if}\quad p=0.
\]
\end{lemma}

\begin{proof}
For any $p\in\R^n$, set
\[
g(t)=dh(tp)\cdot p,\quad t\in[0,1],
\]
then $g\in C^1([0,1],\R)$. Since $g(0)=dh\,(0)=0$, we compute
\begin{align*}
&dh(p)\cdot p=g(1)-g(0)\\
=&\,\int^{1}_{0}\frac{dg}{dt}(t)\ dt\\
=&\,\int^{1}_{0}\langle d^2h(tp)\cdot p,p\rangle\ dt\geq 0.\\
\end{align*}
Since $d^2h(p)$ is positive definite everywhere, the last inequality becomes an equality if and only if $p=0$.
\end{proof}

Now we show that the coordinate function $u:T^\ast M\times\R\rightarrow\R$ serves as a Lyapunov function on $\cY$. Physically, $u$ plays the role of entropy in systems which realize the transfer between mechanical energy and thermal energy.
\begin{theorem}[\textbf{Third Lyapunov function}]\label{Ly-3rd}
$u:T^\ast M\times\R\rightarrow\R$ is monotone increasing along $\Phi^t_H$-orbits in $H^{-1}(0)$. Moreover, $z_0\in H^{-1}(0)$ is non-wandering under $\Phi^t_H$ if and only if $z_0\in\cF_H$.
\end{theorem}

\begin{proof}
Notice that the contact Hamiltonian $H$ satisfies \textbf{(H1)}, we apply Lemma \ref{conv} to $H(x,\cdot,u)$ to find, for any $z\in T^\ast M\times\R$,
\[
\frac{\partial H}{\partial p}(z)\cdot p\geq0\quad\text{and}\quad \frac{\partial H}{\partial p}(z)\cdot p=0\quad\text{if and only if}\quad p=0.
\]

For $z=(x,p,u)\in H^{-1}(0)$,
\begin{equation}\label{ieq:2}
\dot{u}=\frac{\partial H}{\partial p}(z)\cdot p-H(z)=\frac{\partial H}{\partial p}(z)\cdot p\geq0,
\end{equation}
which verifies the first conclusion. From the above discussions, $\dot{u}=0$ implies that
\begin{equation}\label{eq:1}
p=0,\quad \dot{x}=\frac{\partial H}{\partial p}(x,0,u)=0.
\end{equation}

Assume $z_0=(x_0,p_0,u_0)\in H^{-1}(0)$ is non-wandering under $\Phi^t_H$, then along the $\Phi^t_H$-orbits initiating from $z_0,\, \dot{u}\equiv0$. Now \eqref{eq:1} implies that $p\equiv0, \dot{x}\equiv0$ and $z_0\in\cF_H$.
\end{proof}

\subsection{Dynamics on $\cA_H$}
At the beginning, we assume $H$ only satisfies \textbf{(H3)} and show something more general. For two connected, compact, disjoint subsets $\cF_0,\cF_1$ of $\cF_H$, we use $\Sigma(\cF_0,\cF_1)$ to denote all $z\in T^{\ast}M\times\R$ satisfying
\[
\alpha(z)=\cF_0,\,\,\omega(z)=\cF_1,
\]
and $\Sigma_H$ to denote the union of all $\Sigma(\cF_0,\cF_1)$.

\vspace{5pt}
The compactness of $\cA_H$ follows from Theorem \ref{exist-Atr}. Let $z\in\cA_H$, $\cA_H$ is invariant under $\Phi^t_H$, hence
\[
\alpha(z)\subseteq\cA_H,\quad \omega(z)\subseteq\cA_H.
\]
Elementary knowledge from dynamical system shows that $\alpha(z), \omega(z)$ is closed, connected and non-wandering with respect to $\Phi^t_H$. Thus, using Theorem \ref{Ly-3rd}, $\alpha(z)$ and $\omega(z)$ are compact, connected subsets of $\cF_H$ and we conclude that
\[
u|_{\alpha(z)}\equiv u_0,\quad u|_{\omega(z)}\equiv u_1.
\]

\vspace{5pt}
Assume further that $z\in\cA_H\setminus\cF_H$ and $z(t)=(x(t), p(t), u(t))$ is the $\Phi^t_H$-orbit through $z$. By \eqref{ieq:2}, for any $t_0<t_1$,
\[
u(t_1)-u(t_0)=\int_{t_0}^{t_1}\dot{u}\ dt=\int_{t_0}^{t_1}\frac{\partial H}{\partial p}(z(t))\cdot p(t)\ dt>0
\]
since $p(t)$ is not identically $0$ on $[t_0, t_1]$. Thus $u$ is strictly increasing along $z(t)$ and we conclude that
\begin{equation}\label{trans-direction}
u_0<u_1,\quad\,\,\alpha(z)\cap\omega(z)=\emptyset.
\end{equation}
Since the choice of $z\in\cA_H$ is arbitrary, the above discussion leads to
\begin{equation}\label{pro-attractor3}
\cA_H=\cF_H\cup\Sigma_H.
\end{equation}

\vspace{5pt}
\textit{Proof of Theorem B}:
From \textbf{(A2)}, if $H$ satisfies \textbf{(M$_{+}$)} (resp. \textbf{(M$_{-}$)}), then $\alpha(z)\neq\emptyset$ (resp. $\omega(z)\neq\emptyset$) if and only if $z\in\cA_H$. Assume $H$ satisfies \textbf{(H4)}, then $\cF_H$ is a finite set. The connectedness of $\cF_0, \cF_1$ implies that both of them are singleton, so we assume $\cF_0=\{z_0=(x_0,0,u_0)\}, \cF_1=\{z_1=(x_1,0,u_1)\}$. This proves \textbf{(B2)}. It follows from the definition of $\Sigma(\cF_0,\cF_1)$ that
\begin{equation}\label{hetero}
\lim_{t\rightarrow-\infty}z(t)=z_0,\quad \lim_{t\rightarrow\infty}z(t)=z_1,
\end{equation}
and the structure of $\cA_H$ follows from \eqref{pro-attractor3}.

\vspace{5pt}
Since $\cA_H$ is compact, for any two equilibria $z^\prime, z^{\prime\prime}, \{z^\prime,z^{\prime\prime}\}\cup\Sigma(z^{\prime},z^{\prime\prime})$
is closed. Let $z_0\in\cF_H$ and $\mathcal{P}_{z_0}$ the path-component of $\cA_H$ containing $z_0$.  By \eqref{pro-attractor3},\eqref{hetero}, $\mathcal{P}_{z_0}$ consists of finite equilibria and heteroclinic orbits between them and is therefore a union of finite closed sets. Thus $\mathcal{P}_{z_0}$ is a closed subset of $\cA_H$. Since any two path-components are disjoint, by \textbf{(A3)}, there is only one path-component and $\cA_H$ is path-connected. This proves \textbf{(B1)}.

\vspace{5pt}
Finally, notice that \textbf{(B3)} is a direct consequence of \eqref{trans-direction}.

\qed

\vspace{5pt}
\subsection{Applications to discounted systems}
~\\
To apply our results, we consider the discounted Hamiltonian
\begin{equation}\label{dh}
H(x,p,u)=\lb u+h(x,p),\quad\lb>0,
\end{equation}
where $h:T^\ast M\rightarrow\R$ satisfies \textbf{(H1)-(H4)} (these assumptions are independent of $u$). Then $X_H$ (or system \eqref{ch}) could be reduced to the vector field (or discounted system)
\begin{equation}\label{css}
X_{h,\lb}:
\begin{cases}
\dot{x}=\frac{\partial h}{\partial p}(x,p),\\
\dot{p}=-\frac{\partial h}{\partial x}(x,p)-\lb p.
\end{cases}
\end{equation}
defined on $T^\ast M$. Denote the phase flow of $X_{h,\lb}$ by $\phi^t_{h,\lb}:T^\ast M\rightarrow T^\ast M$. Using \textbf{(A1)}, $\phi^t_{h,\lb}$ is forward complete. Let $\Om=d\al$, $X_{h,\lb}$ is also called conformally symplectic since, by \eqref{css}, $\mathcal{L}_{X_{h,\lb}}\Omega=-\lb\Omega$. This leads to
\begin{equation}\label{decay}
(\phi^t_{h,\lb})^\ast\omega=e^{-\lb t}\om\quad\text{for all}\,\,t\geq0.
\end{equation}

\vspace{5pt}
For any $(x,p)\in T^\ast M$, set $(x(t),p(t)):=\phi^t_{h,\lb}(x,p)$ and
\[
u(t)=e^{-\lb t}\bigg[-\frac{h(x,p)}{\lb}+\int^t_0 e^{\lb s}\,\bigg(\frac{\partial h}{\partial p}\cdot p-h\bigg)\,((x(s),p(s))\ ds\bigg],
\]
then $z(t)=(x(t),p(t),u(t))$ satisfies \eqref{ch} with the Hamiltonian \eqref{dh}. Thus by \eqref{energy}, we have for $t\in\R$,
\begin{equation}\label{d-energy}
\lb u(t)+h(x(t),p(t))=0.
\end{equation}
Thus we obtain a converse version of Remark \ref{graph}.

\begin{lemma}
Assume for $(x,p)\in T^\ast M$, there is a compact subset $\cK\subset T^\ast M$ such that $\{(x(t),p(t)):t\in\R\}\subset\cK$, then for $u(t)$ defined above,
\[
(x(t),p(t),u(t))\in\cA_H,\quad\text{for all}\,\,t\in\R.
\]
\end{lemma}

\begin{proof}
It follows directly from \eqref{d-energy} that $u(t),t\in\R$ is bounded, thus the closure of $\{(x(t),p(t),u(t)):t\in\R\}$ is a compact, $\Phi^t_H$-invariant set.
\end{proof}

\vspace{5pt}
Denote the set of equilibria of $\phi^t_{h,\lb}$ by
\[
\cF_{h,\lb}=\{(x_0,0)\in T^{\ast}M\,:\,\partial_{x}h(x_0,0)=0\},
\]
for $(x_0,0),(x_1,0)\in\cF_{h,\lb}$ the set of all $(x,p)\in T^{\ast}M$ satisfying
\[
\lim_{t\rightarrow-\infty}\phi^t_{h,\lb}(x,p)=(x_0,0),\quad \lim_{t\rightarrow+\infty}\phi^t_{h,\lb}(x,p)=(x_1,0).
\]
by $\Sigma(x_0,x_1)$ and
\[
\Sigma_{h,\lb}=\cup_{(x_0,0),(x_1,0)\in\cF_{h,\lb}}\Sigma(x_0,x_1).
\]
Similar to Definition \ref{mga}, one can define the maximal attractor $\cA_{h,\lb}$ for $\phi^t_{h,\lb}$, it is also a maximal compact, $\phi^t_{h,\lb}$-invariant set. Now Remark \ref{graph} and the above lemma help us translate Theorem B into:
\begin{theorem}
Assume $h\in C^2(T^{\ast}M,\R)$ satisfies \textbf{(H1)-(H4)}, then
\begin{itemize}
  \item [\textbf{(1)}] The maximal attractor $\cA_{h,\lb}=\cF_{h,\lb}\cup\Sigma_{h,\lb}$. In particular, $\cA_{h,\lb}$ is path-connected.

  \vspace{5pt}
  \item [\textbf{(2)}] For any $(x,p)\in T^{\ast}M$, both $\alpha(x,p), \omega(x,p)$ consist of at most one equilibrium in $\cF_{h,\lb}$.

  \vspace{5pt}
  \item [\textbf{(3)}] For two distinct equilibria $(x_0,0),(x_1,0)\in\cF_{h,\lb}$, assume $\Sigma(x_0,x_1)\neq\emptyset$, then $h(x_0,0)>h(x_1,0)$.
\end{itemize}
\end{theorem}

\begin{remark}
By \eqref{decay}, $\cA_{h,\lb}$ is of measure zero with respect to the Lebesgue measure $\Omega^n$ on $T^\ast M$.
\end{remark}

We could verify the above theorem on the well-known example
\begin{example}
$H(x,p,u)=\lambda u+\frac{1}{2}p^2+\cos(x),\,\, (x,p,u)\in T^\ast\mathbb{T}\times\R$.
\end{example}

\section{Appendix: Preliminaries}
This section serves as an supplementary explanation of several terms arising in the context. In particular, readers who are not familiar with the notion of viscosity solution may find more information about them.

\vspace{5pt}
\subsection{Vector field $X_H$ and its phase flow}
~\\
Let $X_H$ be the contact Hamiltonian vector field, $X_H$ is $C^1$ by the local expression \eqref{ch}. Thus the local existence theorem \cite[Page 276, Corollary]{Ar1} implies: for every $z_0\in T^\ast M\times\R$, there exists a neighborhood $\cO_0$ of $z_0,\,\,a_{z_0}<0<b_{z_0}$ and a map
\[
\Phi_H\in C^1([a_{z_0}, b_{z_0}]\times\cO_0, T^\ast M\times\R);\,\,(t,z^{\prime})\mapsto\Phi_H(t;z^{\prime})
\]
called the phase flow generated by $X_H$ satisfying for every $t\in[a_{z_0}, b_{z_0}]$,
\begin{itemize}
  \item  $\Phi_H(t\,;\,\cdot):\cO_0\rightarrow T^\ast M\times\R$ is a diffeomorphism onto its image.

  \vspace{5pt}
  \item $\frac{\partial}{\partial t}\Phi_H(t;z^\prime)=X_H(\Phi_H(t;z^\prime))$ and $\Phi_H(0;z^\prime)=z^\prime$.
\end{itemize}
For every $z\in T^{\ast}M\times\R$, let $(a(z), b(z))$ be the maximum existence interval of the integral curve through $z$, the extension theorem \cite[Page 102, Corollary 9]{Ar1} states that $\Phi_H$ is well-defined on some neighborhood of Im($z|_{[a,b]}$), $[a,b]\subset(a(z), b(z))$.

\vspace{5pt}
In the context of this paper, we use the brief notation $\Phi^t_H(\cdot)$ to replace $\Phi_H(t\,;\,\cdot)$. By the above discussion, $\Phi^t_H(z),z\in(a(z), b(z))$ coincides with the unique integral curve $z(t)=(x(t),p(t),u(t))$ of $X_H$ through $z$ and
\begin{prop}\label{exist}
Assume $b(z)<\infty$ (resp. $a(z)>-\infty$), then
\[
\lim_{t\rightarrow b(z)_-}|u(t)|+\|p(t)\|_{x(t)}=\infty\quad(\text{resp.}\,\,\lim_{t\rightarrow a(z)_+}|u(t)|+\|p(t)\|_{x(t)}=\infty).
\]
\end{prop}

\vspace{5pt}
\begin{proof}
It is enough to consider the case $b(z)<\infty$ and we argue by contradiction: assume there is $M>0$ and $t_n<b(z),\,\,n\geq1$ such that
\[
\lim_{n\rightarrow\infty}t_n=b(z)\quad\text{and}\quad\limsup_{n\rightarrow\infty}|u(t_n)|+\|p(t_n)\|_{x(t_n)}\leq M.
\]
Thus, by passing to a subsequence, $z_n=(x(t_n),p(t_n),u(t_n))$ converges to some $z_0$ and by the definition of $b(z)$,
\begin{equation}\label{max-int}
\lim_{n\rightarrow\infty}b(z_n)=\lim_{n\rightarrow\infty}[b(z)-t_n]=0.
\end{equation}
Applying the definition of local phase flow at $z_0$, for $n$ large enough, $\Phi^t(z_n)$ is well-defined on $[a_{z_0}, b_{z_0}]$. This leads to the conclusion $b(z_n)\geq b_{z_0}>0$. This contradicts \eqref{max-int}.
\end{proof}

Applying the above proposition, we have the well-known extension theorem
\begin{corollary}\label{ext}
For $z\in T^{\ast}M\times\R$, if there is a compact subset $\cK\subset T^{\ast}M\times\R$ such that
\[
z(t)\in\cK,\quad\text{for all}\,\,t\in(a(z),b(z)),
\]
then $(a(z),b(z))=\R$.
\end{corollary}

\vspace{5pt}
\subsection{Lipschitz estimate of viscosity solutions}
~\\
Let $H:T^{\ast}M\times\R\rightarrow\R$ be a contact Hamiltonian satisfying \textbf{(H1)-(H2)} and \textbf{(M$_-$)} (resp. \textbf{(M$_+$)}). Since $M$ is compact, the assumption \textbf{(H2)} and \textbf{(M$_{-}$)} (resp. \textbf{(M$_{+}$)}) implies
\begin{itemize}
  \item [\textbf{(H2$^\prime$)}] for every $e, U\in\R$, there is $P\,(e,U)>0$ such that if $\|p\|_x>P,\,\,u\geq U$ (resp. $u\leq U$), then $H(x,p,u)>e$.
\end{itemize}

\vspace{5pt}
It is well-known that (HJ$_-$) does not admit $C^1$ solutions in general. The following definition is originally due to M. Crandall and P. L. Lions \cite{CP} and is used extensively in the study of Hamilton-Jacobi equations.
\begin{definition}\label{vis-def}
Let $u:M\rightarrow\R$ be a continuous function.

\vspace{3pt}
We call $u$ a viscosity sub-solution (resp. super-solution) of (HJ$_-$) if for every $x\in M, \phi\in C^1(M,\R)$ such that $u-\phi$ attains a local maximum (resp. minimum) at $x$,
\[
H(x,d\phi(x),u(x))\leq0\quad(\text{resp. }\,\,H(x,d\phi(x),u(x))\geq0).
\]

\vspace{3pt}
We call $u$ a viscosity solution of (HJ$_-$) if it is both a viscosity sub- and super-solution of (HJ$_-$).
\end{definition}

The following property is standard and crucial in deducing the uniqueness of viscosity solution of the equation (HJ$_-$). Our estimate of solutions also depends on it.
\begin{prop}[\textbf{Comparison principle}]
Assume $H:T^\ast M\times\R\rightarrow\R$ satisfies \textbf{(M$_-$)} and $u,v\in C(M,\R)$ are respectively viscosity sub- and super-solutions of (HJ$_-$). Then $u\leq v$ on $M$.
\end{prop}

\begin{remark}
If $H:T^\ast M\times\R\rightarrow\R$ satisfies \textbf{(M$_+$)}, then $\breve{H}(x,p,u):=H(x,-p,-u)$ satisfies \textbf{(M$_-$)} and the equation (HJ$_-$) for $\breve{H}$ is just (HJ$_+$). Thus the comparison principle also applies to (HJ$_+$).
\end{remark}

To give an estimate of $u_{\pm}$, the idea is to find constant sub- and super-solutions of (HJ$_\pm$) and apply the comparison principle. For $H$ satisfying \textbf{(M$_\pm$)}, it is easily seen that there is $\mathrm{U}\in C^2(M,\R)$ such that
\[
H(x,0,\mathrm{U}(x))=0\quad\text{for any}\,\,\,x\in M.
\]
Assume $H$ satisfies \textbf{(M$_-$)} (resp. \textbf{(M$_+$)}), we define constants
\[
\underline{U}=\min_{x\in M}\mathrm{U}(x),\quad \overline{U}=\max_{x\in M}\mathrm{U}(x),
\]
it follows that
\begin{align*}
H(x,0,\underline{U})\leq&\,\, H(x,0,\mathrm{U}(x))=0\leq H(x,0,\overline{U})\\
(\text{resp.}\,\,H(x,0,-(-\overline{U}))\leq&\,\, H(x,0,\mathrm{U}(x))=0\leq H(x,0,-(-\underline{U}))).
\end{align*}
Thus
\begin{itemize}
  \item $u\equiv \underline{U}$ (resp. $u\equiv-\overline{U}$) is a sub-solution of (HJ$_-$) (resp. (HJ$_+$)),
  \item $u\equiv \overline{U}$ (resp. $u\equiv-\underline{U}$) is a super-solution of (HJ$_-$) (resp. (HJ$_+$)).
\end{itemize}
Note that $-u_+$ is By comparison principle for (HJ$_-$) (resp. (HJ$_+$)),
\begin{equation}\label{0-est}
\underline{U}\leq u_-(x)\leq \overline{U}\quad (\text{resp. }\,\,\underline{U}\leq u_+(x)\leq \overline{U}).
\end{equation}
Combining Definition \ref{vis-def} and Proposition \ref{vis}, for almost every $x\in M$,
\[
H(x,du_-(x),\underline{U})\leq0\quad (\text{resp. }\,\,H(x,du_+(x),\overline{U})\leq0).
\]
Thus by \textbf{(H2$^\prime$)}, for almost every $x\in M$,
\begin{equation}\label{1-est}
\|du_-(x)\|_{x}\leq P(0,\underline{U})\quad (\text{resp. }\,\,\|du_+(x)\|_{x}\leq P(0,\overline{U})).
\end{equation}
Notice that \eqref{0-est} and \eqref{1-est} give the desired estimate.


\begin{thebibliography}{9999}
\bibitem{Ar1} V. I. Arnold, {\it Ordinary differential equations. Translated from the third Russian edition by Roger Cooke}. Springer Textbook. Springer-Verlag, Berlin, 334 pp., 1992.



\bibitem{Ar2} V. I. Arnold, {\it Lectures on partial differential equations. Translated from the second Russian edition by Roger Cooke}. Universitext. Springer-Verlag, Berlin; Publishing House PHASIS, Moscow, x+157 pp., 2004.



\bibitem{Ba3} G. Barles, {\it Solutions de viscosit\'{e} des \'{e}quations de Hamilton-Jacobi}. (French) [Viscosity solutions of Hamilton-Jacobi equations] Math¨¦matiques \& Applications (Berlin) [Mathematics \& Applications], 17. Springer-Verlag, Paris, 1994. x+194 pp.




\bibitem{Br} A. Bravetti, {\it Contact Hamiltonian dynamics: the concept and its ussse}. Entropy \textbf{19}\,: 535, 2017.




\bibitem{BCT} A. Bravetti, H. Cruz, D. Tapias, {\it Contact Hamiltonian mechanics}. Ann. Phys. \textbf{376}\,: 17-39, 2017.




\bibitem{CCL1} R. Calleja, A. Celletti, R.de la Llave, {\it A KAM theory for conformally symplectic systems: efficient algorithms and their validation}. J. Differ. Equ. \textbf{255}\,: 978-1049, 2013.




\bibitem{CCL2} R. Calleja, A. Celletti, R. de la Llave, {\it Local behavior near quasi-periodic solutions of conformally symplectic systems}. J. Dynam. Differential Equations \textbf{25}\, (3): 821-841, 2013.




\bibitem{CCJWY} P. Cannarsa, W. Cheng, L. Jin, K. Wang, J. Yan, {\it Herglotz' variational principle and Lax-Oleinik evolution}. J. Math. Pures Appl. available online, 2020.   http://doi.org/10.1016/j.matpur. 2020.07.002.




\bibitem{CCWY} P. Cannarsa, W. Cheng, K. Wang, J. Yan, {\it Herglotz¡¯ generalized variational principle and contact type Hamilton-Jacobi equations}. Trends in Control Theory and Partial Differential Equations, Springer INdAM Series \textbf{32}\, , 39-67. Springer-Verlag, Berlin, 2019.




\bibitem{CS} P. Cannarsa, C. Sinestrari, {\it Semiconcave functions, Hamilton-Jacobi equations, and optimal control}, Progress in Nonlinear Differential Equations and their Applications. \textbf{58}, Birkh\"auser Boston, Inc., Boston, MA, 2004.




\bibitem{Ca} M. Casdagli, {\it Periodic orbits for dissipative twist maps}. Ergodic Theory Dynam. Systems \textbf{7}\, (2): 165-173, 1987.




\bibitem{C} C. Conley, {\it Isolated invariant sets and the Morse index}. CBMS Regional Conference Series in Mathematics Vol. \textbf{38}, American Mathematical Society, Providence, R.I. iii+89 pp., 1978.




\bibitem{CP} M. G. Crandall, P. L. Lions, {\it Viscosity solutions of Hamilton-Jacobi equations}. Trans. Amer. Math. Soc. \textbf{277}\,:1-42, 1983.





\bibitem{DFIZ} A. Davini, A. Fathi, R. Iturriaga, M. Zavidovique, {\it Convergence of the solutions of the discounted Hamilton-Jacobi equation: convergence of the discounted solutions}. Invent. Math. \textbf{206}: 29-55, 2016.




\bibitem{Ishii} H. Ishii, {\it A short introduction to viscosity solutions and the large time behavior of solutions of Hamilton-Jacobi equations, Hamilton-Jacobi equations: approximations, numerical analysis and applications}, Lecture Notes in Math. Vol. \textbf{2074}, Springer, Heidelberg, pp. 111-249, 2013.




\bibitem{IS} R. Iturriaga, H. Sanchez-Morgado, {\it Limit of the infinite horizon discounted Hamilton-Jacobi equation}. Discrete Contin. Dyn. Syst. Ser. B \textbf{15}\,(3): 623-635, 2011.




\bibitem{LC1} P. Le Calvez, {\it Existence d'orbites quasi-p\'{e}riodiques dans les attracteurs de Birkhoff}. Comm. Math. Phys. \textbf{106}\,(30): 383-394, 1986.




\bibitem{LC2} P. Le Calvez, {\it Propri\'{e}t\'{e}s des attracteurs de Birkhoff}. Ergodic Theory Dynam. Systems \textbf{8}\,(2): 241-310, 1988.




\bibitem{LC3} P. Le Calvez, {\it Dynamical properties of diffeomorphisms of the annulus and of the torus}. SMF/AMS Texts and Monographs, 4. American Mathematical Society, Providence, RI; Socit Mathmatique de France, Paris, x+105 pp., 2000.




\bibitem{LW} C. Liverani, M.P. Wojtkowski, {\it Conformally symplectic dynamics and symmetry of the Lyapunov spectrum}. Commun. Math. Phys. \textbf{194}: 47-60, 1998.




\bibitem{MS} S. Maro and A. Sorrentino, {\it Aubry-Mather theory for conformally symplectic systems}. Comm. Math. Phys. \textbf{354}\,(2): 775-808, 2017.



\bibitem{Mi2} J. Milnor, {\it On the concept of attractor}. Comm. Math. Phys. \textbf{99}\,(2): 177-195, 1985.




\bibitem{Mi3} J. Milnor, {\it Correction and remarks: ``On the concept of attractor''}. Comm. Math. Phys. \textbf{102}\,(3): 517-519, 1985.




\bibitem{R} R. T. Rockafellar, {\it Convex analysis}. Princeton Mathematical Series, No. 28 Princeton University Press, Princeton, N.J. 1970.




\bibitem{SWY} X. Su, L. Wang, J. Yan, {\it Weak KAM theory for Hamilton-Jacobi equations depending on unkown functions}. Discrete Contin. Dyn. Syst. \textbf{36}: 6487-6522, 2016.




\bibitem{WWY1} K. Wang, L. Wang, J. Yan, {\it Implicit variational principle for contact Hamiltonian systems}. Nonlinearity \textbf{30}: 492-515, 2017.




\bibitem{WWY2} K. Wang, L. Wang, J. Yan, {\it Variational principle for contact Hamiltonian systems and its applications}. J. Math. Pures Appl. \textbf{123}\,(9): 167-200, 2019.




\bibitem{WWY3} K. Wang, L. Wang and J. Yan, {\it Aubry-Mather theory for contact Hamiltonian systems}, Comm. Math. Phys. \textbf{366}\,(3): 981-1023, 2019.




\end{thebibliography}
\end{document}